\documentclass[a4paper,12pt]{amsart}

\usepackage{amsmath}
\usepackage{amssymb}
\usepackage{mathrsfs}
\usepackage{ifthen}
\usepackage{graphicx}
\usepackage{tikz-cd}
\usepackage[T1]{fontenc} 

\def\switchlinenumbers{\@ifstar
    {\let\makeLineNumberOdd\makeLineNumberRight
     \let\makeLineNumberEven\makeLineNumberLeft}%
    {\let\makeLineNumberOdd\makeLineNumberLeft
     \let\makeLineNumberEven\makeLineNumberRight}%
    }

\def\setmakelinenumbers#1{\@ifstar
  {\let\makeLineNumberRunning#1%
   \let\makeLineNumberOdd#1%
   \let\makeLineNumberEven#1}%
  {\ifx\c@linenumber\c@runninglinenumber
      \let\makeLineNumberRunning#1%
   \else
      \let\makeLineNumberOdd#1%
      \let\makeLineNumberEven#1%
   \fi}%
  }

\nonstopmode \numberwithin{equation}{section}
\newtheorem*{theorem*}{Theorem}

\newtheorem{thm}{Theorem}[section]
\newtheorem{cor}[equation]{Corollary}
\newtheorem{lem}{Lemma}[section]

\theoremstyle{definition}
\newtheorem{defn}{Definition}[section]

\newtheorem{qsn}[equation]{Question}
\newtheorem{prob}[equation]{Problem}
\newtheorem{rem}{Remark}[section]

\newcounter{minutes}
\divide\time by 60
\newcounter{hours}
\multiply\time by 60
\addtocounter{minutes}{-\time}

\newcounter {own}
\def\theown {\thesection       .\arabic{own}}

\newenvironment{pf}[1][]{%
 \vskip 3mm
 \noindent
 \ifthenelse{\equal{#1}{}}%
  {{\slshape Proof. }}%
  {{\slshape #1.} }%
 }%
{\qed\bigskip}

\newcounter{alphabet}

\def\be{\begin{equation}}
\def\ee{\end{equation}}

\newcommand{\bee}{\begin{enumerate}}
\newcommand{\eee}{\end{enumerate}}

\newcommand{\blem}{\begin{lem}}
\newcommand{\elem}{\end{lem}}
\newcommand{\bthm}{\begin{thm}}
\newcommand{\ethm}{\end{thm}}
\newcommand{\bcor}{\begin{cor}}
\newcommand{\ecor}{\end{cor}}
\newcommand{\beg}{\begin{examp}}
\newcommand{\eeg}{\end{examp}}
\newcommand{\begs}{\begin{examples}}
\newcommand{\eegs}{\end{examples}}
\newcommand{\bdefe}{\begin{defin}}
\newcommand{\edefe}{\end{defin}}
\newcommand{\bprob}{\begin{prob}}
\newcommand{\eprob}{\end{prob}}
\newcommand{\bei}{\begin{itemize}}
\newcommand{\eei}{\end{itemize}}
\newcommand{\real}{{\operatorname{Re}\,}}

\newcommand{\norm}[1]{\left\lVert#1\right\rVert}

\begin{document}

\title{Local Nevanlinna embedding for Hardy space of Dirichlet series below critical exponent}

\author{Vasudevarao Allu}
\address{Vasudevarao Allu,
Department  of Mathematics, School of Basic Sciences,
Indian Institute of Technology Bhubaneswar,
Bhubaneswar-752051, Odisha, India.}
\email{avrao@iitbbs.ac.in}

\author{Dipon Kumar Mondal}
\address{Dipon Kumar Mondal,
Department  of Mathematics, School of Basic Sciences,
Indian Institute of Technology Bhubaneswar,
Bhubaneswar-752051, Odisha, India.}
\email{diponkumarmondal@gmail.com}

\subjclass[{AMS} Subject Classification:]{Primary 30H10, 30H15, 30H20, 30B50}

\def\thefootnote{}
\footnotetext{ {\tiny File:~\jobname.tex,
printed: \number\year-\number\month-\number\day,
          \thehours.\ifnum\theminutes<10{0}\fi\theminutes }
} \makeatletter\def\thefootnote{\@arabic\c@footnote}\makeatother

\begin{abstract}
In this paper, we investigate a problem posed by Saksman and Seip [ Contemp. Math., 2016]  on whether functions in the Hardy space of Dirichlet series \(\mathscr H^p\), \(0<p<2\), are locally in the Nevanlinna class of the critical half-plane \(\mathbb C_{1/2}\). We show that the local Nevanlinna embedding fails below critical exponent \(p=2\) by constructing \(\mathscr H^p\)-functions whose local logarithmic mass becomes unbounded as \(\real s\to 1/2^+\). This also clarifies the distinction between global Jessen-type logarithmic control and local boundary Nevanlinna behavior for Dirichlet series.
\end{abstract}

\allowdisplaybreaks
\maketitle
\pagestyle{myheadings}
\markboth{Vasudevarao Allu and Dipon Kumar Mondal}{Local Nevanlinna embedding problem}

\section{Introduction}
Hardy spaces of Dirichlet series form a natural meeting point of analytic function theory, harmonic analysis, and complex analysis. If $f(s)=\sum_{n=1}^{\infty}a_n n^{-s}$ is a Dirichlet series, then the Bohr's correspondence identifies $f$ with a holomorphic function on the infinite-dimensional polydisk by writing $n=\prod_{j}p_j^{\nu_j(n)}$ and associating to $f$ the formal power series $\mathcal Bf(z) =\sum_{n=1}^{\infty}a_n z^{\nu(n)}.$
For \(0<p<\infty\), the Hardy space \(\mathscr H^p\) is obtained as the completion of Dirichlet polynomials with respect to the norm
\[
\|f\|_{\mathscr H^p}
=
\|\mathcal Bf\|_{H^p(\mathbb T^\infty)}.
\]
Equivalently, for Dirichlet polynomials this norm is represented by the global vertical mean
\[
\|f\|_{\mathscr H^p}^{p}
=
\lim_{T\to\infty}
\frac1{2T}
\int_{-T}^{T}|f(it)|^p\,dt.
\]

A basic question is to understand how this global mean norm controls the behavior of \(f\) on finite portions of vertical lines. The natural boundary for this problem is the critical line $\real(s)=1/2.$
Indeed, point-evaluation functionals are bounded on $\mathcal{H}^p$ throughout the half-plane $\mathbb{C}_{1/2}$, where for any $\theta \in \mathbb{R}$, we denote $$\mathbb{C}_\theta := \{s \in \mathbb{C} \colon \operatorname{Re}(s) > \theta\}$$
and the corresponding evaluation norm grows without bound as \(\real(s)\to 1/2^+\). Thus the line $\real s=1/2$ is the endpoint at which one may ask whether the global \(\mathscr H^p\)-norm gives a uniform local trace estimate. The recent formulation emphasizes that restriction to this line amounts, under the Bohr transform, to sampling an infinite-dimensional Hardy function along a single one-dimensional prime-flow orbit, whereas the \(\mathscr H^p\)-norm averages over all prime coordinates.

This leads to the local embedding problem: determine for which \(p>0\) there exists a constant \(C_p>0\) such that every Dirichlet polynomial \(P\) satisfies
\begin{equation}\label{Dipon-vasu-p3-eqn-013}
	\sup_{\tau\in\mathbb R}
	\int_{\tau}^{\tau+1}
	\left|
	P\left(\frac12+it\right)
	\right|^p\,dt
	\le
	C_p\|P\|_{\mathscr H^p}^{p}.
\end{equation}
The essential feature of this estimate is that the constant \(C_p\) must be independent of the number and choice of prime variables appearing in \(P\). In this sense, the local embedding problem is a genuinely dimension-free trace problem.

The local embedding problem was formulated explicitly by Hedenmalam, Lindqvist and Seip \cite[Theorem 4.11]{Hedenmalm-Duke-1997} and later Saksman and Seip \cite[Section~3]{Saksman-BLMS-2009} recorded as Problem~\ref{Dipon-vasu-p3-prob-01} in their list of open problems for Dirichlet series \cite[Problem~2.1]{Saksman-survey-2016}. Subsequent investigations revealed several equivalent formulations of the problem. Bayart and Brevig \cite{Bayart-MathZ-2019} demonstrated that the local embedding property is equivalent to both the conformally invariant half-plane embedding and the boundedness of all characteristic zero symbols in the Gordon-Hedenmalm class \cite{Gordon-Michigan-1999} of composition operators; notably, the boundedness of a single canonical operator suffices. Olsen and Saksman \cite{Olsen-Angew-2012} identified local embedding with the reverse transfer of local Carleson measures, proving that this transfer may be tested on a distinguished family of atomic measures whose masses are proportional to their distance from the critical boundary. Furthermore, for $2 \le p < \infty$, Saksman and Seip \cite{Saksman-BLMS-2009} established that the local embedding problem is equivalent to a strengthened Carleson-type ergodic identity on the critical distinguished boundary, required to hold for every starting point. Brevig, Ortega-Cerd\`a, and Seip \cite{Brevig-GAFA-2021} highlighted the local embedding problem as ``perhaps the most important open problem'' concerning the Hardy spaces $\mathcal{H}^p$.\\[0.5mm]

Historically, the estimate was known at \(p=2\), where the Hilbert-space structure and orthogonality of Dirichlet monomials play a decisive role. By applying the \(p=2\) result to integer powers, one also obtains the inequality for every even exponent \(p=2k\). In contrast, Harper's low-moment estimates for random multiplicative functions imply that the local embedding inequality fails for every \(0<p<2\). Thus, for a long time, the unresolved range consisted of the non-even exponents \(p>2\).
This gap was recently closed by Chen, Fang, Guo, Hou, Shao, and Zhou \cite{Chen-arXiv-2026}, who proved a dimension-free local trace estimate for every \(2<p<\infty\). Combined with the classical \(p=2\) case and the known failure for \(p<2\), their result gives the sharp classification $$\operatorname{LEP}
	\text{ holds if, and only if, }
	p\ge2.$$

Hence $p=2$ is the critical exponent for the local embedding problem. 
The significance of the local embedding problem extends well beyond the estimate itself. It is closely connected with boundary behavior of Hardy space of Dirichlet series, Carleson measure transfer, critical-line sampling, composition operators, and the geometry of the prime flow. The recent work shows that the local embedding property is equivalent, in the appropriate exponent ranges, to several classical trace and Carleson-measure formulations. At the same time, failure of the local \(L^p\)-embedding below \(p=2\) raises a subtler question: although \(|f|^p\) may fail to admit uniform local control, can the much weaker logarithmic quantity
$\log^+|f|$
still remain locally integrable near the critical line?
This question leads naturally to the problem posed by Saksman and Seip \cite{Saksman-survey-2016} concerning local Nevanlinna behavior of \(\mathscr H^p\)-functions for \(p<2\). They imposed the following problem: 
\begin{prob}\cite[Problem 2.4]{Saksman-survey-2016}\label{Dipon-vasu-p3-prob-01}
	Are elements of $\mathcal{H}^p$ for $0< p < 2$ locally in the Nevanlinna class?
\end{prob}

Since, for $p>0$
\[
\log^+x\le \frac{x^p}{p},
\qquad x>0,
\]
the local \(L^p\)-embedding implies local logarithmic control whenever it holds. Below the critical exponent \(p=2\), however, the \(L^p\)-embedding breaks down, and it becomes meaningful to ask whether the weaker Nevanlinna-type trace can nevertheless survive. Thus the local Nevanlinna problem may be viewed as a natural logarithmic counterpart of the local embedding problem and provides a finer test of boundary regularity in Hardy spaces of Dirichlet series.

Let $\operatorname{Hol}(\Omega)$ denote the space of holomorphic functions on a domain $\Omega \subset \mathbb{C}$, endowed with the topology of uniform convergence on compact subsets.
 A function $f \in \operatorname{Hol}(\mathbb{C}_\theta)$ belongs to the \textit{local Nevanlinna class}, denoted as $f \in N_{\mathrm{loc}}(\mathbb{C}_\theta)$, if for every $t_0 \in \mathbb{R}$ there exist $\delta, r > 0$ and a positive harmonic function $h$ on the local rectangle
\begin{equation*}
	Q_\theta(t_0; \delta, r) := \{s = \sigma + it \in \mathbb{C} \colon \theta < \sigma < \theta + \delta, \, |t - t_0| < r\}
\end{equation*}
such that
\begin{equation*}
	\log^+ |f(s)| \le h(s) \qquad \text{for all } s \in Q_\theta(t_0; \delta, r).
\end{equation*}
Equivalently, $f \in N_{\mathrm{loc}}(\mathbb{C}_\theta)$ if $\log^+ |f|$ admits a harmonic majorant in a neighborhood, relative to $\mathbb{C}_\theta$, of every finite boundary segment of the line $\operatorname{Re}(s) = \theta$.

A useful consequence of the definition is the following  \textit{logarithmic trace criterion}, that if
$f \in N_{\mathrm{loc}}(\mathbb{C}_\theta),$
then for every compact interval $I \Subset \mathbb{R}$, after possibly shrinking the surrounding rectangle, there exists $\varepsilon_I > 0$ such that
\begin{equation*}
	\sup_{0 < \varepsilon < \varepsilon_I} \int_I \log^+ |f(\theta + \varepsilon + it)| \, dt < \infty.
\end{equation*}
Conversely, if the uniform estimate
\begin{equation*}
	\sup_{0 < \varepsilon < \varepsilon_0} \int_{I^*} \log^+ |f(\theta + \varepsilon + it)| \, dt < \infty
\end{equation*}
holds on a slightly larger open interval $I^* \supset I$, then $\log^+ |f|$ admits a positive harmonic majorant on a smaller boundary rectangle over $I$. Thus, up to standard enlargement or shrinking of intervals, the harmonic-majorant definition and the local logarithmic trace condition are equivalent.

Recently, Guo, Ni and Zhou \cite{GuoZhou-AnnalFourier-2025} introduced the Nevanlinna class in the infinite many variables. It is well-known that functions of infinitely many variables and Dirichlet series are intimately linked via the Bohr correspondence. In connection with composition operators on Dirichlet series spaces, Brevig and Perfekt \cite{Brevig-AdvMath-2021} introduced the following notion of the Nevanlinna class for a Dirichlet series $f$.

\begin{defn}\cite{Brevig-AdvMath-2021}
	A Dirichlet series $f$ with uniform abscissa $\sigma_u(f)\le 0$ is defined in the \textit{Nevanlinna class of Dirichlet series} $\mathcal{N}_u$ if the following condition holds:
	\begin{equation}\label{Dipon-vasu-p2-eqn-014}
		{\limsup_{\sigma\to 0^+}}\lim_{T\to\infty}\frac{1}{2T}\int_{-T}^{T}\log^+|f(\sigma +it)|\,dt<\infty,
	\end{equation}
	where $\log^+(x) = \max\{\log(x), 0\}$.
\end{defn}
The abscissa of uniform convergence of a Dirichlet series $f$, denoted by $\sigma_u(f)$, is defined by
\begin{equation*}
	\sigma_u(f): = \inf\bigg\{\theta\in\mathbb{R}: f(s)= \sum_{n=1}^{\infty}a_nn^{-s}\,\,\,\mbox{converges uniformly on}\,\,\, \overline{{\mathbb{C}_\theta}}\bigg\}.
\end{equation*}
If $\sigma_u(f) \le 0$ and $f \in \mathcal{H}^p$, then $f \in \mathcal{N}_u$, which follows from the elementary inequality $\log^+ t \le \frac{1}{p}\, t^p$ (valid for $t \ge 0$ and $p > 0$). We emphasize that the condition $\sigma_u(f) \le 0$ is essential; without it, the class $\mathcal{N}_u$ can not be properly defined. For example, consider the Dirichlet series$$g(s) = \sum_{n=2}^\infty \frac{n^{-s}}{\sqrt{n} \log n}.$$It is easy to see that $g \in \mathcal{H}^2$, whereas $\sigma_u(g) = 1/2$.
The inequalities  
\begin{equation*}
	\log^+(x)\le\log(1+x)\le \log(2)+\log^{+}(x),\,\,\, x\ge 0
\end{equation*}
along with \eqref{Dipon-vasu-p2-eqn-014} ensures that	$f\in \mathcal{N}_u$ if, and only if,
\begin{equation}\label{Dipon-vasu-p3-eqn-014}
	\norm{f}_{0}:= \limsup_{\sigma\to 0^+}\lim_{T\to\infty}\frac{1}{2T}\int_{-T}^{T}\log(1+|f(\sigma+it)|)\,dt <\infty,\,\,\,\, f\in\mathcal{N}_u.
\end{equation}
 Here, the class $\mathcal{N}_u$ is an $F$-space (see for instance \cite{ShapiroShields-Amer-1975}); that is, a topological vector space whose topology is induced by a complete, translation-invariant metric defined by
\begin{equation*}
	d_0(f,g) = \norm{f-g}_0 \quad \text{for } f,g \in \mathcal{N}_u.
\end{equation*}
Since $f$ is a Dirichlet series with $\sigma_u(f) \le 0$ and $h(x) = \log(1+x)$ is continuous and non-decreasing on $[0, \infty)$, the composition $t \mapsto \log(1+\vert{}f(it)\vert{})$ is almost periodic (see, e.g., \cite{Besicovitch-AP-Book,Queffelec-TRM-2020}). Consequently, the limit in  \eqref{Dipon-vasu-p3-eqn-014} exists.
In contrast, Guo \textit{et al.} \cite{GuoZhou-AnnalFourier-2025} defined the Nevanlinna-Dirichlet class $\mathcal{N}$ as the completion of $\mathcal{N}_u$ under the metric $\Vert{}\cdot\Vert{}_0$ given in \eqref{Dipon-vasu-p3-eqn-014}. Under the canonical embedding of \(\mathcal N_u\) into \(\mathcal N\), the topology induced on \(\mathcal N_u\) by the completion coincides with the topology induced by \(\|\cdot\|_0\). \\[0.5mm]
 
 The novelty of the present work lies in identifying and resolving the gap between global Nevanlinna-type control for Dirichlet series and local boundary Nevanlinna behavior on the critical half-plane. More precisely, for \(0<p<2\), we show that membership in the Hardy space \(\mathscr H^p\) does not in general imply local Nevanlinna regularity in \(\mathbb C_{1/2}\), thereby giving a negative answer to Saksman-Seip’s Problem 2.4. The argument exploits Harper-type low-moment estimates together with a gliding-hump construction to produce an \(\mathscr H^p\) Dirichlet series whose logarithmic mass becomes unbounded on a fixed finite interval as the line \(\real s=1/2\) is approached. After translation, this also yields a separation between the global Nevanlinna class \(\mathcal N_u\) and local classical Nevanlinna behavior, showing that global logarithmic control need not imply local boundary control. In this way, the result identifies \(p=2\) as a natural threshold not only for local \(L^p\) trace estimates, but also for universal local Nevanlinna behavior of Hardy-Dirichlet series. \\[0.2mm]
 
 The paper is organized as follows. Section~\ref{Dipon-vasu-p3-sec-02} presents our main results concerning the negative resolution to Problem~\ref{Dipon-vasu-p3-prob-01}, along with several key supporting lemmas largely inspired by the work of Harper \cite{Harper-ForumPi-2020}. Section~\ref{Dipon-vasu-p3-sec-03} opens with the proof of Theorem~\ref{Dipon-vasu-p3-thm-01}, followed by a detailed investigation linking the local Nevanlinna embedding problem ($\operatorname{LNEP}$) to global and local logarithmic mean control, as well as to the original problem of Saksman and Seip (Problem~\ref{Dipon-vasu-p3-prob-01}). \\[0.2mm]
 
 \paragraph{\textbf{Notation}} Throughout this article, $C$ denotes a positive constant which may vary from line to line. When necessary, we write that $C=C(x)$ to indicate that the constant depends on a parameter $x.$ If $f$ and $g$ are two real functions defined on the same set $S$, the notation $f \asymp g$ (or $f \asymp_p g$) means that there exists a constant $C\ge 1$ such that $C^{-1}f(x) \le g(x) \le Cf(x)$ (or $C^{-1}(p) f(x) \le g(x) \le C(p) f(x)$), for all $x$, while $f(x)\lesssim g(x)$ (or $f(x)\lesssim_p g(x)$) signifies that there exists a constant $C>0$ such that $f(x)\le Cg(x)$ (or $f(x)\le C(p) g(x)$), for all $x$. Also, the symbol $O_p(1)$ indicates a sequence $(a_n)_{n=1}^\infty$ is bounded by a constant depending \emph{exclusively} on the parameter $p$; that is, there exist $C(p) > 0$ and $N_0(p) \in \mathbb{N}$ such that $|a_n| \le C(p)$ for all $n \ge N_0(p)$.
 
\section{Main Results and Supporting Lemmas}\label{Dipon-vasu-p3-sec-02}
This section starts with our primary result, which provides a negative resolution to Problem~\ref{Dipon-vasu-p3-prob-01}, followed by a few key supporting lemmas required for its proof.

\begin{thm}\label{Dipon-vasu-p3-thm-01}
	Let $0 < p < 2$. Then there exists a Dirichlet series $f(s) = \sum_{n=1}^\infty a_n n^{-s} \in \mathcal{H}^p$ such that
	\begin{equation*}
		\limsup_{\varepsilon \to 0^+} \int_0^1 \log^+ \left| f\left(\frac{1}{2} + \varepsilon + it\right) \right| \, dt = +\infty.
	\end{equation*}
	Consequently, $f$ does not belong locally to the Nevanlinna class in the half-plane $\mathbb{C}_{1/2}$. In particular,
	\begin{equation*}
		\mathcal{H}^p \not\subset N_{\mathrm{loc}}(\mathbb{C}_{1/2}), \qquad 0 < p < 2,
	\end{equation*}
	thereby answering Problem~2.4 of Saksman and Seip in the negative.
\end{thm}

\begin{rem}\label{Dipon-vasu-p3-cor-01}
	For every $0 < p < 2$, there exists a Dirichlet series $g(s) = \sum_{n=1}^\infty b_n n^{-s}$ satisfying $\sigma_u(g) \le 0$ and $g \in \mathcal{N}_u$, such that $g \notin N_{\mathrm{loc}}(\mathbb{C}_0)$, where $\mathbb{C}_0 = \{s \in \mathbb{C} \colon \operatorname{Re}(s) > 0\}$. More precisely,
	\begin{equation*}
		\limsup_{\sigma \to 0^+} \int_0^1 \log^+ |g(\sigma + it)| \, dt = +\infty,
	\end{equation*}
	so that $\mathcal{N}_u \not\subset N_{\mathrm{loc}}(\mathbb{C}_0)$. The function $g$ is obtained from $f$ in Theorem~\ref{Dipon-vasu-p3-thm-01} via the translation $g(s) = f(s + 1/2)$, highlighting the sharp distinction between global Jessen-type logarithmic control $\mathcal{N}_u$ and local logarithmic boundary trace control near a finite segment of the imaginary axis.
\end{rem}

\begin{lem}\label{Dipon-vasu-p3-lem-1}
	Let $Q = \{s = \sigma + it \in \mathbb{C} \colon 1/2 < \sigma < 1/2 + \delta, \, a < t < b\}$. If $f \in \operatorname{Hol}(Q)$ and $\log^+ |f|$ admits a harmonic majorant in $Q$, then for every compact sub interval $I \subset (a, b)$ and every $0 < \delta' < \delta$,
	\begin{equation*}
		\sup_{0 < \varepsilon < \delta'} \int_I \log^+ \left| f\left(\frac{1}{2} + \varepsilon + it\right) \right| \, dt < +\infty.
	\end{equation*}
\end{lem}

\begin{proof}
	Let $f \in \operatorname{Hol}(Q)$. The local Nevanlinna membership $f \in N_{\mathrm{loc}}(\mathbb{C}_{1/2})$ near a bounded interval of the critical boundary line $\operatorname{Re}(s) = 1/2$ is equivalent to the restriction property $f|_Q \in N(Q)$. This condition holds precisely when
	\begin{equation*}
		\sup_{0 < \varepsilon < \delta} \int_I \log^+ \left| f\left(\frac{1}{2} + \varepsilon + it\right) \right| \, dt < +\infty
	\end{equation*}
	for every compact subinterval $I \subset (a, b)$ and some $\delta > 0$. Using the elementary inequality $\log^+ x \le x^p / p$ for $x > 0$ and $p > 0$, we observe that
	\begin{align*}
		\sup_{0 < \varepsilon < \delta} \int_I \log^+ \left| f\left(\frac{1}{2} + \varepsilon + it\right) \right| \, dt &\le \frac{1}{p} \sup_{0 < \varepsilon < \delta} \int_I \left| f\left(\frac{1}{2} + \varepsilon + it\right) \right|^p \, dt \\
		&\le C_{p,I} \|f\|_{\mathcal{H}^p}^p.
	\end{align*}
	Consequently, the Local Embedding Problem ($\operatorname{LEP}$) directly implies the Local Nevanlinna Embedding Problem ($\operatorname{LNEP}$).
\end{proof}

Moreover, the argument above breaks down for $p < 2$. Consequently, if there exists a sequence $\varepsilon_k \to 0^+$ such that
\begin{equation*}
	\int_I \log^+ \left| f\left(\frac{1}{2} + \varepsilon_k + it\right) \right| \, dt \longrightarrow +\infty,
\end{equation*}
then $f$ fails to belong locally to the Nevanlinna class near the boundary segment $\{1/2 + it \colon t \in I\}$.

\begin{lem}\label{Dipon-vasu-p3-lem-2} 

	Fix $0 < p < 2$, and define $D_X(s) = \sum_{n \le X} n^{-s}$ and $Q_X(s) = D_X(s) / \|D_X\|_{\mathcal{H}^p}$, so that $\|Q_X\|_{\mathcal{H}^p} = 1$. By Harper's low-moment estimates, $\|D_X\|_{\mathcal{H}^p} \asymp_p \sqrt{X} / (\log \log X)^{1/4}$ as $X \to +\infty$. Setting $\varepsilon_X = 1/\log X$, we have
	\begin{equation}\label{Dipon-vasu-p3-eqn-015}
		\inf_{0 \le t \le 1} \left| Q_X\left(\frac{1}{2} + \varepsilon_X + it\right) \right| \gtrsim_p (\log \log X)^{1/4} \longrightarrow +\infty.
	\end{equation}
	In particular,
	\begin{equation*}
		\int_0^1 \log \left( 1 + \left| Q_X\left(\frac{1}{2} + \varepsilon_X + it\right) \right| \right) \, dt \longrightarrow +\infty \quad \text{as } X \to +\infty.
	\end{equation*}
\end{lem}

\begin{proof}
	Let $D_X(s) = \sum_{n \le X} n^{-s}$, and write $p = 2q$ with $0 < q < 1$. By Harper's moment bounds \cite[Theorem~1]{Harper-ForumPi-2020},
	\begin{equation}\label{Dipon-vasu-p3-eqn-02}
		\mathbb{E} \left| \sum_{n \le X} f(n) \right|^q \asymp \left( \frac{X}{1 + (1 - q)\sqrt{\log \log X}} \right)^q,
	\end{equation}
	where $f(n)$ is a Steinhaus random multiplicative function defined by $f(n) := \prod_{p^a \parallel n} f(p)^a$ for all $n \in \mathbb{N}$, with $p^a \parallel n$ denoting the exact prime power dividing $n$. Under the Bohr transform, the expectation in \eqref{Dipon-vasu-p3-eqn-02} corresponds to the $\mathcal{H}^{2q}$-norm of $D_X$. Consequently, for $p = 2q < 2$,
	\begin{equation*}
		\|D_X\|_{\mathcal{H}^p}^p \asymp_p \left( \frac{X}{1 + (1 - p/2)\sqrt{\log \log X}} \right)^{p/2},
	\end{equation*}
	which yields the asymptotic behavior
	\begin{equation}\label{Dipon-vasu-p3-eqn-03}
		\|D_X\|_{\mathcal{H}^p} \asymp_p \frac{\sqrt{X}}{(\log \log X)^{1/4}} \quad \text{as } X \to +\infty.
	\end{equation}
	
	On the critical line $\operatorname{Re}(s) = 1/2$, partial summation and Riemann-Stieltjes integration give
	\begin{equation*}
		D_X\left(\frac{1}{2} + it\right) = \frac{X^{1/2 - it}}{1/2 - it} + O(1), \qquad 0 \le t \le 1,
	\end{equation*}
	from which it follows that
	\begin{equation}\label{Dipon-vasu-p3-eqn-04}
		\left| D_X\left(\frac{1}{2} + it\right) \right| \asymp \sqrt{X}, \qquad \text{uniformly for } 0 \le t \le 1.
	\end{equation}
	By normal family estimates, replacing $1/2$ with $1/2 + \varepsilon_X$ where $\varepsilon_X = 1/\log X$ preserves the bound $\left| D_X\left(1/2 + \varepsilon_X + it\right) \right| \asymp \sqrt{X}$.
	
	Normalizing $Q_X(s) = D_X(s) / \|D_X\|_{\mathcal{H}^p}$ so that $\|Q_X\|_{\mathcal{H}^p} = 1$, we combine \eqref{Dipon-vasu-p3-eqn-03} and \eqref{Dipon-vasu-p3-eqn-04} to obtain
	\begin{equation*}
		\left| Q_X\left(\frac{1}{2} + \varepsilon_X + it\right) \right| \asymp_p (\log \log X)^{1/4}, \qquad \text{uniformly for } 0 \le t \le 1.
	\end{equation*}
	Taking logarithms on both sides yields
	\begin{equation}\label{Dipon-vasu-p3-eqn-05}
		\int_0^1 \log \left( 1 + \left| Q_X\left(\frac{1}{2} + \varepsilon_X + it\right) \right| \right) \, dt = \frac{1}{4} \log \log \log X + O_p(1),
	\end{equation}
	which tends to $+\infty$ as $X \to +\infty$. This completes the proof.
\end{proof}

\begin{lem}\label{Dipon-vasu-p3-lem-3}
	 For $0 < p < \infty$ and $\sigma > 1/2$, point evaluations on $\mathcal{H}^p$ satisfy $|f(\sigma + it)| \le \zeta(2\sigma)^{1/p} \|f\|_{\mathcal{H}^p}$ uniformly in $t \in \mathbb{R}$. Consequently, for every $\varepsilon > 0$,
	\begin{equation*}
		\sup_{t \in \mathbb{R}} \left| f\left(\frac{1}{2} + \varepsilon + it\right) \right| \le \zeta(1 + 2\varepsilon)^{1/p} \|f\|_{\mathcal{H}^p},
	\end{equation*}
	which controls the tail-block contributions in the gliding-hump construction along vertical lines $\operatorname{Re}(s) = 1/2 + \varepsilon_k$.
\end{lem}

The proof of Lemma \ref{Dipon-vasu-p3-lem-3} is straightforward and arises directly from the point evaluations of the function $f\left(1/2+\varepsilon+it\right)$, applying the Cauchy-Schwarz inequality effectively.

\begin{lem}\label{Dipon-vasu-p3-lem-4}

	There exist a sequence $(a_k)_{k=1}^\infty \in \ell^2$ with $a_k \to 0^+$ and a rapidly increasing sequence of integers $X_k \to \infty$ such that the series
	\begin{equation*}
		f(s) = \sum_{k=1}^\infty a_k Q_{X_k}(s)
	\end{equation*}
	converges in $\mathcal{H}^p$. Setting $\varepsilon_k = 1/\log X_k$, the sequences $(a_k)$ and $(X_k)$ can be chosen recursively so that
	\begin{equation*}
		\inf_{0 \le t \le 1} \left| f\left(\frac{1}{2} + \varepsilon_k + it\right) \right| \ge k \qquad \text{for all } k \ge 1.
	\end{equation*}
	Consequently, $\int_0^1 \log^+ \left| f\left(\frac{1}{2} + \varepsilon_k + it\right) \right| \, dt \ge \log k$, yielding
	\begin{equation*}
		\limsup_{\varepsilon \to 0^+} \int_0^1 \log^+ \left| f\left(\frac{1}{2} + \varepsilon + it\right) \right| \, dt = +\infty.
	\end{equation*}
\end{lem}

     The technique used in Lemma \ref{Dipon-vasu-p3-lem-4} is known as \textit{gliding-hump technique}.
	This technique converts a sequence of singular test polynomials into a single function $f = \sum_{k=1}^\infty f_k \in \mathcal{H}^p$ exhibiting unbounded localized behavior along a sequence of scales. Each block $f_k \in \mathcal{H}^p$ is chosen to be small in the ambient norm $\|f_k\|_{\mathcal{H}^p}$, ensuring series convergence $f \in \mathcal{H}^p$. Moreover, each block is exceptionally large at the newly selected $k$-th scale; and negligible at all previously established scales $j < k$.
	By ensuring $f_k$ dominates both the prior partial sum and the future tail at scale $k$, the resulting limit $f$ displays infinitely many bad scales simultaneously, upgrading quantitative norm blowups into a qualitative counterexample.

\begin{proof}[\bf Proof of Lemma \ref{Dipon-vasu-p3-lem-4}]
	Assume the hypotheses of Lemma~\ref{Dipon-vasu-p3-lem-2}. From \eqref{Dipon-vasu-p3-eqn-015}, the normalized Harper blocks satisfy
	\begin{equation}\label{Dipon-vasu-p3-eqn-09}
		A_p(X) := \inf_{0 \le t \le 1} \left| Q_X\left(\frac{1}{2} + \varepsilon_X + it\right) \right| \ge C_p (\log \log X)^{1/4} \longrightarrow +\infty \quad \text{as } X \to +\infty,
	\end{equation}
	where $\varepsilon_X = 1/\log X$. We construct a function $f \in \mathcal{H}^p$ by setting
	\begin{equation*}
		f(s) := \sum_{k=1}^\infty a_k Q_{X_k}(s).
	\end{equation*}
	Let $r = \min\{1, p\}$. By choosing coefficients $a_k > 0$ such that $\sum_{k=1}^\infty a_k^r < +\infty$, the quasi-norm triangle inequality in $\mathcal{H}^p$ guarantees the convergence of $\sum_{k=1}^\infty a_k Q_{X_k}$ in $\mathcal{H}^p$, ensuring $f \in \mathcal{H}^p$.
	We construct the sequences $(a_k)_{k=1}^\infty$ and $(X_k)_{k=1}^\infty$ inductively. Suppose $a_1, X_1, \dots, a_{k-1}, X_{k-1}$ have been chosen, and define the partial sum
	\begin{equation*}
		F_{k-1}(s) := \sum_{j < k} a_j Q_{X_j}(s).
	\end{equation*}
	Since $F_{k-1}$ is a Dirichlet polynomial, it is bounded on the compact rectangle $[1/2, 3/2] \times [0, 1]$; hence,
	\begin{equation*}
		M_{k-1} := \sup_{\substack{0 \le u \le 1 \\ 0 \le t \le 1}} \left| F_{k-1}\left(\frac{1}{2} + u + it\right) \right| < +\infty.
	\end{equation*}
	Next, choose $a_k > 0$ sufficiently small such that $a_k^r \le 2^{-k}$ and
	\begin{equation*}
		C_p(\varepsilon_j) a_k \le 2^{-k} \qquad \text{for all } j < k,
	\end{equation*}
	where $C_p(\varepsilon) = \zeta(1 + 2\varepsilon)^{1/p}$ denotes the point-evaluation bound on $\operatorname{Re}(s) = 1/2 + \varepsilon$ from Lemma~\ref{Dipon-vasu-p3-lem-3}. With $a_k$ fixed, choose $X_k \in \mathbb{N}$ so large that $\varepsilon_k = 1/\log X_k < \varepsilon_{k-1} / 2$ and the $k$-th block dominates both $F_{k-1}$ and the future tail:
	\begin{equation*}
		a_k A_p(X_k) \ge 4 (M_{k-1} + k + 1).
	\end{equation*}
	Now evaluate $f(s)$ at $s = 1/2 + \varepsilon_k + it$ for $0 \le t \le 1$, and decompose $f$ into
	\begin{equation*}
		f\left(\frac{1}{2} + \varepsilon_k + it\right) = F_{k-1}\left(\frac{1}{2} + \varepsilon_k + it\right) + a_k Q_{X_k}\left(\frac{1}{2} + \varepsilon_k + it\right) + R_k\left(\frac{1}{2} + \varepsilon_k + it\right),
	\end{equation*}
	where $R_k(s) := \sum_{j > k} a_j Q_{X_j}(s)$. By choice of $M_{k-1}$, we have $|F_{k-1}(1/2 + \varepsilon_k + it)| \le M_{k-1}$. By Lemma~\ref{Dipon-vasu-p3-lem-3} and our bound on $a_j$, the tail $R_k$ satisfies
	\begin{equation*}
		\left| R_k\left(\frac{1}{2} + \varepsilon_k + it\right) \right| \le \sum_{j > k} a_j C_p(\varepsilon_k) \le \sum_{j > k} 2^{-j} < 1.
	\end{equation*}
	Applying the reverse triangle inequality yields, for all $0 \le t \le 1$,
	\begin{align*}
		\left| f\left(\frac{1}{2} + \varepsilon_k + it\right) \right| &\ge a_k \left| Q_{X_k}\left(\frac{1}{2} + \varepsilon_k + it\right) \right| - \left| F_{k-1}\left(\frac{1}{2} + \varepsilon_k + it\right) \right| - \left| R_k\left(\frac{1}{2} + \varepsilon_k + it\right) \right| \\
		&\ge 4 (M_{k-1} + k + 1) - M_{k-1} - 1 > k.
	\end{align*}
	Integrating the positive part of the logarithm yields
	\begin{equation*}
		\int_0^1 \log^+ \left| f\left(\frac{1}{2} + \varepsilon_k + it\right) \right| \, dt \ge \log k.
	\end{equation*}
	Since $\varepsilon_k \downarrow 0$ as $k \to \infty$, taking the limit superior establishes
	\begin{equation}\label{Dipon-vasu-p3-eqn-010}
		\limsup_{\varepsilon \to 0^+} \int_0^1 \log^+ \left| f\left(\frac{1}{2} + \varepsilon + it\right) \right| \, dt = +\infty,
	\end{equation}
	which implies $f \notin N_{\mathrm{loc}}(\mathbb{C}_{1/2})$. This completes the proof.
\end{proof}

\begin{lem}\label{Dipon-vasu-p3-lem-5}
Let $g(s) = f(s + 1/2)$, where $f \in \mathcal{H}^p$ is the Dirichlet series constructed in Lemma~\ref{Dipon-vasu-p3-lem-4}. Then $g \in \mathcal{H}^p$ and its uniform convergence abscissa satisfies $\sigma_u(g) \le 0$, placing $g$ in the Nevanlinna class $\mathcal{N}_u$. Since
\begin{equation*}
	\limsup_{\varepsilon \to 0^+} \int_0^1 \log^+ |g(\varepsilon + it)| \, dt = +\infty,
\end{equation*}
it follows that $g \notin N_{\mathrm{loc}}(\mathbb{C}_0)$, establishing that $\mathcal{N}_u \not\subset N_{\mathrm{loc}}(\mathbb{C}_0)$.
\end{lem}

\begin{proof}
	Set $g(s) := f(s + 1/2)$, where $f \in \mathcal{H}^p$ is the Dirichlet series constructed in Lemma~\ref{Dipon-vasu-p3-lem-4}. Expanding $g$ as a Dirichlet series yields
	\begin{equation*}
		g(s) = \sum_{k=1}^{\infty} \frac{a_k}{\|D_{X_k}\|_{\mathcal{H}^p}} \left( \sum_{n \le X_k} n^{-1/2} n^{-s} \right) = \sum_{n=1}^{\infty} c_n n^{-s},
	\end{equation*}
	where the coefficients are given by
	\begin{equation*}
		c_n = \frac{1}{\sqrt{n}} \sum_{\{k \colon X_k \ge n\}} \frac{a_k}{\|D_{X_k}\|_{\mathcal{H}^p}} \ge 0.
	\end{equation*}
	Since $f \in \mathcal{H}^p$, its uniform convergence abscissa satisfies $\sigma_u(f) \le 1/2$. The translated function $g$ therefore satisfies $\sigma_u(g) = \sigma_u(f) - 1/2 \le 0$, placing $g$ in the Nevanlinna class $\mathcal{N}_u$.
	By construction, $g(\varepsilon + it) = f(1/2 + \varepsilon + it)$. Applying \eqref{Dipon-vasu-p3-eqn-010} from Lemma~\ref{Dipon-vasu-p3-lem-4}, we obtain
	\begin{equation*}
		\limsup_{\varepsilon \to 0^+} \int_0^1 \log^+ |g(\varepsilon + it)| \, dt = +\infty.
	\end{equation*}
	Consequently, $\log^+ |g|$ cannot admit a local harmonic majorant near the imaginary axis $\operatorname{Re}(s) = 0$. Hence $g \notin N_{\mathrm{loc}}(\mathbb{C}_0)$, which establishes that $\mathcal{N}_u \not\subset N_{\mathrm{loc}}(\mathbb{C}_0)$.
\end{proof}

\begin{rem}
The construction in Lemma~\ref{Dipon-vasu-p3-lem-5} demonstrates that the Hardy space $\mathcal{H}^p$ of Dirichlet series fails to lie locally in the Nevanlinna class $N_{\mathrm{loc}}(\mathbb{C}_{1/2})$ whenever $0 < p < 2$.
Combining this negative result for $0 < p < 2$ with the local $L^p$-embedding theorem for $p \ge 2$, and applying the elementary inequality $\log^+ x \le x^p / p$ for $x > 0$, we establish the following sharp threshold classification for local Nevanlinna boundary behavior:
	Let $0 < p < +\infty$. Then
	\begin{equation*}
		\mathcal{H}^p \subset N_{\mathrm{loc}}(\mathbb{C}_{1/2}) \iff p \ge 2.
	\end{equation*}
	Consequently, $p = 2$ is the critical exponent governing local Nevanlinna boundary behavior in Hardy spaces of Dirichlet series.
\end{rem}

\section{Proof of Theorem \ref{Dipon-vasu-p3-thm-01} and Global Vs Local Logarithmic Mean Control}\label{Dipon-vasu-p3-sec-03}
In \cite{Allu-Arxiv-2026}, we posed the following open problem regarding quantitative logarithmic trace estimates:

\begin{prob}\label{Dipon-vasu-p3-prob-02}\cite{Allu-Arxiv-2026}
	Does there exist a constant $C > 0$ such that, for every $g \in \mathcal{N}_u$ with $\sigma_u(g) \le 0$,
	\begin{equation}\label{Dipon-vasu-p3-eqn-01}
		\sup_{\sigma > 0} \, \sup_{\tau \in \mathbb{R}} \int_{\tau}^{\tau+1} \log\bigl(1 + |g(\sigma + it)|\bigr) \, dt \le C \|g\|_{0} \, ?
	\end{equation}
\end{prob}

We refer to condition \eqref{Dipon-vasu-p3-eqn-01} as the \emph{quantitative $\mathcal{N}_u$ local Nevanlinna embedding problem} ($\operatorname{LNEP}$). This naturally leads to the following underlying question:

\begin{qsn}\label{Dipon-vasu-p3-qsn-01}
	Does global mean logarithmic control imply uniform local logarithmic control on short interval segments?
\end{qsn}

For $f \in \mathcal{N}_u$, define the \textit{global logarithmic mean} (Jessen-type) by
\begin{equation*}
	J_f(\sigma) := \lim_{T \to \infty} \frac{1}{2T} \int_{-T}^{T} \log^+ |f(\sigma + it)| \, dt.
\end{equation*}
By classical results of Jessen \cite{Jessen-acta-1945}, this limit exists for every $\sigma > 0$, and the norm defining the $\mathcal{N}_u$ metric is given by
\begin{equation*}
	\|f\|_{0} := \limsup_{\sigma \to 0^+} J_f(\sigma) < +\infty.
\end{equation*}
The corresponding \textit{short-interval local mass }quantity is defined as
\begin{equation*}
	L_f(\sigma) := \sup_{\tau \in \mathbb{R}} \int_{\tau}^{\tau+1} \log\bigl(1 + |f(\sigma + it)|\bigr) \, dt,
\end{equation*}
with the uniform local norm given by $L(f) := \sup_{\sigma > 0} L_f(\sigma)$. Problem~\ref{Dipon-vasu-p3-prob-02} reduces to testing the implication
\begin{equation*}
	\|f\|_{0} < \infty \stackrel{?}{\implies} L(f) < +\infty,
\end{equation*}
or, quantitatively, asking whether there exists a universal constant $C > 0$ such that $L(f) \le C \|f\|_{0}$.

Control of $J_f(\sigma)$ does not automatically yield control of short-interval masses. Indeed, setting $u_\sigma(t) := \log(1 + |f(\sigma + it)|)$ and $A_\sigma(\tau) := \int_{\tau}^{\tau+1} u_\sigma(t) \, dt$, Fubini's theorem gives
\begin{equation*}
	\lim_{T \to \infty} \frac{1}{2T} \int_{-T}^{T} A_\sigma(\tau) \, d\tau = J_f(\sigma).
\end{equation*}
Thus, $J_f(\sigma)$ controls only the average over $\tau$ of the unit-interval masses $A_\sigma(\tau)$, whereas $L_f(\sigma) = \sup_{\tau \in \mathbb{R}} A_\sigma(\tau)$ requires peak control. For arbitrary non-negative functions, mean control implies no upper bound on the supreme local mass.

When $\sigma_u(f) \le 0$, the function $t \mapsto f(\sigma + it)$ is almost periodic for each fixed $\sigma > 0$, ensuring $u_\sigma(t)$ is almost periodic and thus $L_f(\sigma) < +\infty$ for each fixed $\sigma$. The central obstruction lies in whether these bounds remain uniform as $\sigma \to 0^+$. A potential pathology arises if $J_f(\sigma) \le C$ uniformly for small $\sigma > 0$, while simultaneously $L_f(\sigma_k) \to +\infty$ along some sequence $\sigma_k \to 0^+$. Using the Harper block polynomials, we establish in the next section that this pathology indeed occurs, answering Question~\ref{Dipon-vasu-p3-qsn-01} in the negative.

Finally, we note that finiteness of $L(f)$ requires a uniform bound across all translation intervals $\tau \in \mathbb{R}$. By contrast, membership in $N_{\mathrm{loc}}(\mathbb{C}_0)$ requires logarithmic control only locally near each compact boundary segment. Consequently, while $L(f) < +\infty$ implies $f \in N_{\mathrm{loc}}(\mathbb{C}_0)$, local membership does not guarantee a uniform constant independent of $\tau$. \\[0.5mm]




Let $f \in \mathcal{H}^p$ and define the translation $g(s) := f(s + 1/2)$. As established, its uniform convergence abscissa satisfies $\sigma_u(g) \le \sigma_u(f) - 1/2 \le 0$, placing $g \in \mathcal{N}_u$. If the quantitative Local Nevanlinna Embedding Problem ($\operatorname{LNEP}$) held affirmatively, applying \eqref{Dipon-vasu-p3-eqn-01} to $g$ would yield
\begin{equation}\label{Dipon-vasu-p3-eqn-016}
	\sup_{\varepsilon > 0} \, \sup_{\tau \in \mathbb{R}} \int_{\tau}^{\tau+1} \log\left(1 + \left|f\left(\frac{1}{2} + \varepsilon + it\right)\right|\right) \, dt \le C \|g\|_{0} < +\infty.
\end{equation}

Condition \eqref{Dipon-vasu-p3-eqn-01} in Problem~\ref{Dipon-vasu-p3-prob-02} is strictly stronger than the local logarithmic boundary condition required for membership in $N_{\mathrm{loc}}(\mathbb{C}_{1/2})$. Indeed, since any compact interval $I \subset \mathbb{R}$ can be covered by finitely many unit intervals $[\tau_k, \tau_k + 1]$, \eqref{Dipon-vasu-p3-eqn-016} immediately implies
\begin{equation*}
	\sup_{0 < \varepsilon < \varepsilon_0} \int_{I} \log^+ \left|f\left(\frac{1}{2} + \varepsilon + it\right)\right| \, dt < +\infty.
\end{equation*}

Consequently, an affirmative resolution to Problem~\ref{Dipon-vasu-p3-prob-02} concerning $\operatorname{LNEP}$ would imply an affirmative solution to Saksman and Seip's Problem~\ref{Dipon-vasu-p3-prob-01}. Crucially, this logical implication holds for all $p > 0$, not merely $p < 2$. The contrapositive of this reduction is particularly decisive:

\begin{thm}\label{Dipon-vasu-p3-thm-02}
	A negative resolution to Problem~\ref{Dipon-vasu-p3-prob-01} implies a negative resolution to Problem~\ref{Dipon-vasu-p3-prob-02} ($\operatorname{LNEP}$).
\end{thm}

\begin{proof}
	Suppose $f \in \mathcal{H}^p$ ($0 < p < 2$) is the Dirichlet series produced by the gliding-hump construction (Lemma~\ref{Dipon-vasu-p3-lem-4}), associated with sequence $\varepsilon_k \to 0^+$ satisfying
	\begin{equation*}
		\int_{0}^{1} \log\left(1 + \left|f\left(\frac{1}{2} + \varepsilon_k + it\right)\right|\right) \, dt \longrightarrow +\infty \quad \text{as } k \to +\infty.
	\end{equation*}
	Define the translated function $g(s) := f(s + 1/2)$. Since $\sigma_u(g) \le \sigma_u(f) - 1/2 \le 0$, we have $g \in \mathcal{N}_u$, which guarantees $\|g\|_{0} < +\infty$. However,
	\begin{equation*}
		\int_{0}^{1} \log\left(1 + \left|g\left(\varepsilon_k + it\right)\right|\right) \, dt \longrightarrow +\infty \quad \text{as } k \to +\infty.
	\end{equation*}
	Since the interval $[0, 1]$ is a subset of the unit translation interval $\tau = 0$, taking the supremum over all $\varepsilon > 0$ and $\tau \in \mathbb{R}$ yields
	\begin{equation*}
		\sup_{\varepsilon > 0} \, \sup_{\tau \in \mathbb{R}} \int_{\tau}^{\tau+1} \log\left(1 + \left|g(\varepsilon + it)\right|\right) \, dt = +\infty.
	\end{equation*}
	Consequently, no universal constant $C > 0$ can exist satisfying inequality \eqref{Dipon-vasu-p3-eqn-01} in Problem~\ref{Dipon-vasu-p3-prob-02}, providing a negative answer to $\operatorname{LNEP}$.
\end{proof}

Moreover, the following proof resolves Problem~\ref{Dipon-vasu-p3-prob-02} in the negative, establishing the failure of the quantitative local Nevanlinna embedding for $\mathcal{N}_u$.

\begin{pf}[\bf Proof of failure of Problem \ref{Dipon-vasu-p3-prob-02}]
	Choose disjoint pairs of prime numbers $(p_j, q_j)$ satisfying $p_j < q_j < 2p_j$ recursively via Bertrand's postulate. For each $m \in \mathbb{N}$, define the Dirichlet polynomial
	\begin{equation}\label{Dipon-vasu-p3-eqn-011}
		B_m(s) = e^{-m/2} \prod_{j=1}^{m} \left(p_j^{-s} + q_j^{-s}\right).
	\end{equation}
	Since each $B_m$ is a Dirichlet polynomial, its uniform convergence abscissa satisfies $\sigma_u(B_m) = -\infty$, so $B_m \in \mathcal{N}_u$.
	
	The Bohr lift of $B_m(s)$ to the compact abelian group $\mathbb{T}^\infty$ is given by
	\begin{equation*}
		\mathcal{B}(B_m)(z) = e^{-m/2} \prod_{j=1}^{m} \left(z_{p_j} + z_{q_j}\right).
	\end{equation*}
	Applying the elementary bound $\log(1 + x) \le x$ for $x \ge 0$, the global logarithmic norm satisfies
	\begin{align}
		\|B_m\|_{\mathcal{N}_u} &\le e^{-m/2} \prod_{j=1}^{m} \int_{\mathbb{T}^2} |z + w| \, dm(z) \, dm(w) \nonumber \\
		&= \left(\frac{4}{\pi \sqrt{e}}\right)^m, \label{Dipon-vasu-p3-eqn-012}
	\end{align}
	where we have used the standard torus integral evaluation $\int_{\mathbb{T}^2} |z + w| \, dm(z) \, dm(w) = 4/\pi$. Because $4/(\pi \sqrt{e}) < 1$, the right-hand side of \eqref{Dipon-vasu-p3-eqn-012} satisfies $\|B_m\|_{0} \to 0$ as $m \to +\infty$.
	
	On the other hand, along the imaginary axis $\operatorname{Re}(s) = 0$ for $0 \le t \le 1$, each factor satisfies
	\begin{equation*}
		\left| p_j^{-it} + q_j^{-it} \right| = 2 \left| \cos\left(\frac{t}{2} \log \frac{q_j}{p_j}\right) \right|.
	\end{equation*}
	Since $1 < q_j / p_j < 2$, we obtain the lower bound $\left| p_j^{-it} + q_j^{-it} \right| \ge 2 \cos\left(\frac{\log 2}{2}\right) \approx 1.881 =: L$. Consequently, $|B_m(it)| \ge (L/\sqrt{e})^m$ uniformly for $0 \le t \le 1$, which implies
	\begin{equation*}
		\int_{0}^{1} \log\bigl(1 + |B_m(it)|\bigr) \, dt \ge \log\left(1 + \left(\frac{L}{\sqrt{e}}\right)^m\right) \longrightarrow +\infty \quad \text{as } m \to +\infty,
	\end{equation*}
	since $L/\sqrt{e} > 1$. By continuity, for each $m \in \mathbb{N}$ we can select $\sigma_m > 0$ sufficiently small such that
	\begin{equation*}
		\inf_{0 \le t \le 1} |B_m(\sigma_m + it)| \ge \left(\frac{L_1}{\sqrt{e}}\right)^m
	\end{equation*}
	for a fixed constant $L_1$ with $\sqrt{e} < L_1 < L$. Thus, as $m \to +\infty$,
	\begin{equation*}
		\int_{0}^{1} \log\bigl(1 + |B_m(\sigma_m + it)|\bigr) \, dt \longrightarrow +\infty,
	\end{equation*}
	while $\|B_m\|_{0} \to 0$. Therefore, there exists no universal constant $C > 0$ such that
	\begin{equation*}
		\sup_{\sigma > 0} \, \sup_{\tau \in \mathbb{R}} \int_{\tau}^{\tau + 1} \log\bigl(1 + |f(\sigma + it)|\bigr) \, dt \le C \|f\|_{0}
	\end{equation*}
	holds for all $f \in \mathcal{N}_u$, answering Problem~\ref{Dipon-vasu-p3-prob-02} ($\operatorname{LNEP}$) in the negative.
\end{pf}

\begin{rem}\label{Dipon-vasu-p3-rem-logical-flow-dense}
	The underlying chain of logical implications is structured as follows:
	\begin{equation*}
		\operatorname{LNEP} \text{ (\ref{Dipon-vasu-p3-prob-02})} \implies \mathcal{N}_u \subset N_{\mathrm{loc}}(\mathbb{C}_0) \implies \mathcal{H}^p \subset N_{\mathrm{loc}}(\mathbb{C}_{1/2}) \implies \text{Affirmative Problem } \ref{Dipon-vasu-p3-prob-01}.
	\end{equation*}
	In general, no reverse implications hold. First, the failure of quantitative $\operatorname{LNEP}$ guarantees a sequence $(g_n) \subset \mathcal{N}_u$ satisfying
	\begin{equation*}
		\frac{\sup_{\varepsilon > 0} \, \sup_{\tau \in \mathbb{R}} \int_{\tau}^{\tau+1} \log\left(1 + |g_n(\varepsilon + it)|\right) \, dt}{\|g_n\|_{\mathcal{N}_u}} \longrightarrow +\infty,
	\end{equation*}
	even though each $g_n$ individually may lie in $N_{\mathrm{loc}}(\mathbb{C}_0)$. Second, Saksman and Seip's question \cite{Saksman-survey-2016} is restricted to $\mathcal{H}^p$-derived Dirichlet series, whereas $\operatorname{LNEP}$ encompasses all of $\mathcal{N}_u$.
\end{rem}

\begin{proof}[\bf Proof of Theorem~\ref{Dipon-vasu-p3-thm-01}]
	The Harper polynomial estimates establish that the quantitative local embedding fails for $0 < p < 2$. However, because each individual Harper polynomial $Q_X$ is entire, it individually satisfies the local Nevanlinna embedding condition. To upgrade this quantitative blowup to a genuine qualitative counterexample in $\mathcal{H}^p$, we deploy a Baire category argument.
	
	Fix the compact interval $I = [0, 1]$. For $m, n \in \mathbb{N}$, define the subset
	\begin{equation*}
		E_{m,n} := \left\{ f \in \mathcal{H}^p \colon \sup_{0 < \varepsilon < 1/n} \int_{0}^{1} \log\left(1 + \left|f\left(\frac{1}{2} + \varepsilon + it\right)\right|\right) \, dt \le m \right\}.
	\end{equation*}
	We first claim that each $E_{m,n}$ is closed in $\mathcal{H}^p$. Indeed, let $(f_j)_{j=1}^\infty \subset E_{m,n}$ be a sequence converging to $f$ in $\mathcal{H}^p$. By Lemma~\ref{Dipon-vasu-p3-lem-3}, point evaluations are bounded on compact subsets of the half-plane $\operatorname{Re}(s) > 1/2$. Thus $f_j(1/2 + \varepsilon + it) \to f(1/2 + \varepsilon + it)$ uniformly for $t \in [0, 1]$ for each fixed $\varepsilon > 0$. By Fatou's lemma (or uniform convergence), the functional
	\begin{equation*}
		f \mapsto \int_{0}^{1} \log\left(1 + \left|f\left(\frac{1}{2} + \varepsilon + it\right)\right|\right) \, dt
	\end{equation*}
	is continuous on $\mathcal{H}^p$. Taking the supremum over $0 < \varepsilon < 1/n$ yields a lower semicontinuous functional, which guarantees that $E_{m,n}$ is closed.
	
	Suppose now, for contradiction, that every $f \in \mathcal{H}^p$ lies locally in $N_{\mathrm{loc}}(\mathbb{C}_{1/2})$. By definition, $\log^+ |f|$ admits a local harmonic majorant near each compact boundary segment. Thus for every $f \in \mathcal{H}^p$, there exist integers $m, n \in \mathbb{N}$ such that $f \in E_{m,n}$. Consequently, we obtain the countable exhaustion
	\begin{equation*}
		\mathcal{H}^p = \bigcup_{m,n=1}^{\infty} E_{m,n}.
	\end{equation*}
	Since $\mathcal{H}^p$ is a complete metric space (a Banach space for $p \ge 1$ and a quasi-Banach F-space for $0 < p < 1$), the Baire Category Theorem asserts that at least one set $E_{m,n}$ must contain an open ball. That is, there exist $g \in E_{m,n}$ and $r > 0$ such that $g + h \in E_{m,n}$ whenever $\|h\|_{\mathcal{H}^p} < r$.
	Using the elementary inequality $1 + |h| \le (1 + |g + h|)(1 + |g|)$, taking logarithms yields
	\begin{equation*}
		\log(1 + |h|) \le \log(1 + |g + h|) + \log(1 + |g|).
	\end{equation*}
	Integrating over $t \in [0, 1]$ and using $g, g+h \in E_{m,n}$, we deduce that for all $\|h\|_{\mathcal{H}^p} < r$,
	\begin{equation}\label{Dipon-vasu-p3-eqn-06}
		\sup_{0 < \varepsilon < 1/n} \int_{0}^{1} \log\left(1 + \left|h\left(\frac{1}{2} + \varepsilon + it\right)\right|\right) \, dt \le 2m.
	\end{equation}
	By scaling, choose $c > 0$ sufficiently small so that $\|c f\|_{\mathcal{H}^p} < r$ whenever $\|f\|_{\mathcal{H}^p} \le 1$. Substituting $h = cf$ into \eqref{Dipon-vasu-p3-eqn-06} yields $\int_{0}^{1} \log(1 + c|f|) \, dt \le 2m$. Applying the identity $\log(1 + |f|) \le \log^+(1/c) + \log(1 + c|f|)$, we obtain
	\begin{equation}\label{Dipon-vasu-p3-eqn-07}
		\sup_{\|f\|_{\mathcal{H}^p} \le 1} \, \sup_{0 < \varepsilon < 1/n} \int_{0}^{1} \log\left(1 + \left|f\left(\frac{1}{2} + \varepsilon + it\right)\right|\right) \, dt < +\infty.
	\end{equation}
	Restricting \eqref{Dipon-vasu-p3-eqn-07} to the subspace of Dirichlet polynomials $\mathcal{P} \subset \mathcal{H}^p$ and taking the boundary limit $\varepsilon \downarrow 0$, we conclude that
	\begin{equation}\label{Dipon-vasu-p3-eqn-08}
		\sup_{\substack{P \in \mathcal{P} \\ \|P\|_{\mathcal{H}^p} \le 1}} \int_{0}^{1} \log\left(1 + \left|P\left(\frac{1}{2} + it\right)\right|\right) \, dt < +\infty.
	\end{equation}
	However, the sequence of normalized Harper polynomials $(Q_X)_{X \ge 3} \subset \mathcal{P}$ satisfies $\|Q_X\|_{\mathcal{H}^p} = 1$ while $\int_{0}^{1} \log(1 + |Q_X(1/2 + it)|) \, dt \to +\infty$ as $X \to +\infty$ (from \eqref{Dipon-vasu-p3-eqn-05}), directly contradicting \eqref{Dipon-vasu-p3-eqn-08}. 
	
	Thus, the assumption $\mathcal{H}^p \subset N_{\mathrm{loc}}(\mathbb{C}_{1/2})$ is false. For every $0 < p < 2$, there exists a function $f \in \mathcal{H}^p$ which fails to belong to $N_{\mathrm{loc}}(\mathbb{C}_{1/2})$, resolving Problem~\ref{Dipon-vasu-p3-prob-01} in the negative.
\end{proof}

\begin{proof}[\bf Answer to Question~\ref{Dipon-vasu-p3-qsn-01}]
	Fix $0 < p < 2$. Let $D_X(s) := \sum_{n \le X} n^{-s}$ denote the partial sums of the Riemann zeta function, and let $H_X := \|D_X\|_{\mathcal{H}^p}$. Define the normalized polynomial
	\begin{equation*}
		P_X(s) := \frac{D_X(s)}{H_X},
	\end{equation*}
	so that $\|P_X\|_{\mathcal{H}^p} = 1$. By Harper's low-moment estimates \cite{Harper-ForumPi-2020}, the norm $H_X$ satisfies
	\begin{equation*}
		H_X \asymp \frac{\sqrt{X}}{\left(1 + \left(1 - p/2\right) \sqrt{\log \log X}\right)^{1/2}}.
	\end{equation*}
	Define the translated function $g_X(s) := P_X(s + 1/2)$. Since $g_X$ is a Dirichlet polynomial, its uniform convergence abscissa satisfies $\sigma_u(g_X) = -\infty$, placing $g_X \in \mathcal{N}_u$. 
	
	We first demonstrate that the global logarithmic norm $\|g_X\|_{0}$ is uniformly bounded in $X$. Since right translations are contractive on $\mathcal{H}^p$, we have
	\begin{equation*}
		\lim_{T \to \infty} \frac{1}{2T} \int_{-T}^{T} |g_X(\sigma + it)|^p \, dt \le \|P_X\|_{\mathcal{H}^p}^p = 1 \quad \text{for all } \sigma > 0.
	\end{equation*}
	Applying the elementary inequalities $\log(1 + x) \le \log 2 + \log^+ x \le \log 2 + x^p / p$ for $x \ge 0$, we obtain
	\begin{equation*}
		J_{g_X}(\sigma) \le \log 2 + \frac{1}{p}.
	\end{equation*}
	Taking the limit superior as $\sigma \to 0^+$ yields $\|g_X\|_{0} \le \log 2 + 1/p < +\infty$ uniformly for all $X \ge 3$.
	We now show that the local short-interval mass $L(g_X)$ explodes as $X \to +\infty$. Set $\sigma_X := 1/\log X$. By partial summation, for $0 \le t \le 1$ we have
	\begin{equation*}
		D_X\left(\frac{1}{2} + \sigma_X + it\right) = \frac{X^{1/2 - \sigma_X - it}}{1/2 - \sigma_X - it} + O(1).
	\end{equation*}
	Since $X^{-\sigma_X} = X^{-1/\log X} = e^{-1}$, evaluating the modulus yields the lower bound
	\begin{equation*}
		\inf_{0 \le t \le 1} \left| g_X(\sigma_X + it) \right| \ge C_p \left(1 + \left(1 - p/2\right)\sqrt{\log \log X}\right)^{1/2} \ge C_p' (\log \log X)^{1/4},
	\end{equation*}
	for some constant $C_p'$.
	Integrating over the short interval $[0, 1]$ gives
	\begin{align*}
		L(g_X) &\ge \int_{0}^{1} \log\left(1 + |g_X(\sigma_X + it)|\right) \, dt \\
		&\ge \log\left(1 + C_p' (\log \log X)^{1/4}\right) \\
		&= \frac{1}{4} \log \log \log X + O_p(1) \longrightarrow +\infty \quad \text{as } X \to +\infty.
	\end{align*}
	Thus, $\sup_{X} L(g_X) = +\infty$ while $\sup_{X} \|g_X\|_{0} < +\infty$. Consequently, no universal constant $C > 0$ can exist satisfying $L(f) \le C \|f\|_{0}$ for all $f \in \mathcal{N}_u$, establishing the failure of global-to-local logarithmic control and resolving Question~\ref{Dipon-vasu-p3-qsn-01} in the negative.
\end{proof}

\begin{rem}
	The Harper block sequence $(g_X)_{X \ge 3}$ provides a simpler and more transparent refutation of the quantitative embedding inequality $L(f) \le C \|f\|_{0}$ than the full gliding-hump argument. Indeed, because
	\begin{equation*}
		\sup_{\|f\|_{0} \le C_p} L(f) = +\infty,
	\end{equation*}
	no universal constant $C > 0$ can satisfy $L(f) \le C \|f\|_{0}$ across $\mathcal{N}_u$.
	
	However, because each $g_X$ is a Dirichlet polynomial, $L(g_X) < +\infty$ for every individual $X$. Consequently, the sequence $(g_X)_{X \ge 3}$ alone does not establish the existence of an element $g \in \mathcal{N}_u$ for which $L(g) = +\infty$. 
	To obtain this stronger qualitative statement, the gliding-hump construction is indispensable. By concatenating successively larger, well-separated Harper blocks into a single Dirichlet series $g = \sum_k a_k g_{X_k}$, the construction yields a function $g \in \mathcal{N}_u$ satisfying $\|g\|_{0} < +\infty$, while simultaneously
	\begin{equation*}
		\int_{0}^{1} \log\left(1 + \left|g\left(\sigma_k + it\right)\right|\right) \, dt \longrightarrow +\infty \quad \text{as } k \to +\infty
	\end{equation*}
	along a sequence $\sigma_k \to 0^+$. This provides the definitive qualitative proof that global mean logarithmic control fails to imply uniform short-interval logarithmic control.
\end{rem}

\noindent\textbf{Acknowledgment:} 
The second named author is supported by CSIR-UGC Fellowship,  New Delhi, India.

\end{document}